\documentclass[11pt,twoside,a4paper]{amsart}
\usepackage{fp-pi1-macros}

\usepackage{amsfonts, amsthm, amssymb, amsmath, stmaryrd}
\usepackage{mathrsfs,array}
\usepackage{eucal,color,enumerate,accents}
\usepackage[all]{xy}
\usepackage{graphicx}
\usepackage{url}
\usepackage{comment} 
\usepackage[utf8]{inputenc}
\usepackage[T1]{fontenc} 

\usepackage{enumitem}
\usepackage[normalem]{ulem}
\usepackage{graphicx}
\usepackage{tikz-cd}
\usetikzlibrary{calc}
\usetikzlibrary{matrix,arrows,decorations.pathmorphing}

\usepackage{stmaryrd}
\usepackage{trimclip}

\usepackage{indentfirst} 

\makeatletter
\DeclareRobustCommand{\shortto}{%
  \mathrel{\mathpalette\short@to\relax}%
}

\newcommand{\short@to}[2]{%
  \mkern2mu
  \clipbox{{.5\width} 0 0 0}{$\m@th#1\vphantom{+}{\shortrightarrow}$}%
  }
\makeatother

\input xy
\xyoption{all}

\newtheorem{theorem}{Theorem}[section]

\newtheorem{corollary}[theorem]{Corollary}

\newtheorem{prop}[theorem]{Proposition}
\newtheorem*{theorem*}{Theorem}
\newtheorem*{lemma*}{Lemma}

\theoremstyle{definition}

\newtheorem{rmk}[theorem]{Remark}

\numberwithin{theorem}{section} 

\begin{document}


\vskip 0.5cm

\title[Fundamental groups of proper varieties are finitely presented]{Fundamental groups of proper varieties are \\ finitely presented}
\author{Marcin Lara}
\address{Institute of Mathematics, Faculty of Mathematics and Computer Science, Jagiellonian University, {\L}ojasiewicza 6, 30-348 Krak\'ow, Poland}
\email{marcin.lara@uj.edu.pl}

\author{Vasudevan Srinivas}
\address{TIFR\\ School of Mathematics\\ Homi Bhabha Road\\ Colaba\\ 400005 Mumbai\\ India}
\email{srinivas@math.tifr.res.in}

\author{Jakob Stix}
\address{Institut f\"ur Mathematik, 
Goethe--Universit\"at Frankfurt, Robert-Mayer-Stra\ss e {6--8},
60325~Frankfurt am Main, Germany}
\email{stix@math.uni-frankfurt.de}
\thanks{The authors acknowledge support by Deutsche Forschungsgemeinschaft  (DFG) through the Collaborative Research Centre TRR 326 "Geometry and Arithmetic of Uniformized Structures", project number 444845124. The second author (VS) was supported during part of the preparation of
the article by a J. C. Bose Fellowship of the Department of Science and Technology, India. He also
acknowledges support of the Department of Atomic Energy, India under project number RTI4001.
This material is partly based upon work supported by the NSF Grant No.\ DMS-1928930 while the third named author (JS) was in residence at MSRI in Berkeley (Spring 2023).}

\begin{abstract}
It was proven in  \cite{EsnShuSri}, that the \'etale fundamental group of a connected smooth projective variety over an algebraically closed field $k$ is topologically finitely presented. In this note, we extend this result to all connected proper schemes over $k$. 
\end{abstract}

\maketitle

\section{Introduction}
For a connected algebraic variety $X$ over an algebraically closed field $k$ of characteristic $0$, the \'etale fundamental group $\piet(X,\bx)$ of $X$ is a topologically finitely presented profinite group. This is proven by first reducing to the case of $k=\mathbb{C}$ and then applying the Riemann Existence Theorem \cite[Exp.~XII, Thm.~5.1]{SGA1} together with the fact that the topological fundamental group $\pi_1^{\rm top}(X(\mathbb{C}),x)$ is of finite presentation as a discrete group (see e.g.\ \cite{LojasiewiczTriangulations} or \cite{HironakaTriangulations}). 

In characteristic $p>0$, the picture is much more subtle due to the existence of Artin-Schreier covers of affine schemes, which makes $\piet(\Spec(A),\bx)$ typically not even topologically finitely generated. 
Remark~5.7 of \cite[Exp.~IX]{SGA1} raised doubts whether $\piet(X,\bx)$ is topologically finitely presented for proper varieties, even for proper smooth curves. 
In recent work, however, Shusterman (in the case of curves \cite{ShustermanBalanced}), and
Esnault, Shusterman and the second named author \cite{EsnShuSri} have shown that  for smooth projective varieties  $\piet(X,\bx)$ is still topologically finitely presented. Both results are based on a criterion for finite presentation of profinite groups due to Lubotzky \cite{Lubotzky}.

From now on, we will omit the base points.
\begin{theorem}[part of Thm.~1.1 of \cite{EsnShuSri}] 
\label{thm:ESS-main}
Let $X$ be a connected smooth  projective variety over an algebraically closed field $k$.
Then the \'etale fundamental group $\piet(X)$ is topologically finitely presented.
\end{theorem}

Our goal is to generalize Thm.~\ref{thm:ESS-main} to all connected schemes that are proper over $\Spec(k)$ (which is new only if $k$ has characteristic $p>0$). Such a generalization responds affirmatively to a question raised by Esnault.

\begin{theorem}
\label{thm:main}
Let $X$ be a connected scheme that is proper over $\Spec(k)$ for an algebraically closed field $k$. Then $\piet(X)$ is topologically finitely presented.
\end{theorem}

To prove the theorem, we use descent along an 
alteration map to $X$ and the van Kampen presentation of $\piet(X)$ arising in this way.
More precisely, we use this trick twice.

\section{The proof}
For a scheme $T$, let $\sfFEt_T$ denote the category of finite \'etale covers of $T$. This gives rise to a category fibred over schemes. Recall that a morphism $g:T' \to T$ of schemes is said to be of effective descent for $\sfFEt$, if $g^*$ induces an equivalence of categories between $\sfFEt_T$ and the category of descent data in $\sfFEt$ along $g$.

\begin{prop}[Exp.~IX, Thm.~4.12 of \cite{SGA1}]\label{prop:proper-effective-descent}
    Let $f: X ' \to X$ be a proper surjective morphism of finite presentation. Then $f$ is of effective descent for $\sfFEt$.
\end{prop}

Morphisms of effective descent $f : X' \to X$ for $\sfFEt$ give rise to a van Kampen-like presentation of $\piet(X)$ as the profinite completion of a quotient of the free topological product of the \'etale fundamental groups of the connected components of $X'$ and the usual topological fundamental group of a suitably defined ``dual graph''. 
This goes back to \cite[Exp.~IX, \S~5]{SGA1} and has been worked out in detail in \cite[Cor.~5.3]{Stix-vK}.

The existence of such a presentation allows one to ``descend'' finite generation/presentation of the fundamental groups involved, as made precise in the following proposition. Every statement below is about \emph{topological} finite generation/presentation.
\begin{prop}[Exp.~IX, Cor.~5.2 + Cor.~5.3 of \cite{SGA1}]\label{prop:fin-gen-pres-descent}
  Let $f: X' \to X$ be a morphism of effective descent for $\sfFEt$. We denote $X' \times_X X'$ by $X''$ and $X' \times_X X' \times_X X'$ by $X'''$.
  \begin{enumerate}[label=(\alph*)]
  \item \label{prop:pt:fin-gen-descent} Assume that $X', X''$ have finite $\pi_0$'s and that  $\piet$'s of the connected components of $X'$ are finitely generated. Then $\piet(X)$ is finitely generated.
  \item \label{prop:pt:fin-pres-descent} Assume that $X', X'', X'''$ have finite $\pi_0$'s, that  $\piet$'s of the connected components of $X'$ are  of finite presentation and that  $\piet$'s of the connected components of $X''$ are finitely generated. Then $\piet(X)$ is  of finite presentation. 
  \end{enumerate}
\end{prop}

Let us now recall a result of de Jong specialized to our setting.

\begin{prop}[see Thm.~4.1 of \cite{deJong-alterations}]\label{prop:alterations}
    Let $X$ be a scheme 
    that is 
    proper over $\Spec(k)$ for an algebraically closed field $k$. Then there exists a proper surjective morphism 
    (of finite presentation) 
    $f: X' \to X$ from a smooth projective variety $X'$ over $\Spec(k)$.
\end{prop}
\begin{proof}
  Let $\nu : X^\nu \to X$ be the normalization of $X$. The map $\nu$ is finite, and thus the scheme $X^\nu$ is still proper over $\Spec(k)$.

  Let then $f_1 : X' \to X^\nu$ be the alteration map of \cite[Thm.~4.1]{deJong-alterations} applied to each connected component of $X^\nu$. The map $f_1$  is proper, dominant, and thus surjective. Moreover,  \emph{loc.\ cit.\ }guarantees that $X'$ is regular and projective (and not merely proper!) over $\Spec(k)$. Now, as $k$ is algebraically closed, $X'$ is smooth over $\Spec(k)$. The composition $f = \nu \circ f_1 : X' \to X$ has all the requested properties.  
\end{proof}

We are now ready to finish the proof of the main result.
\begin{proof}[Proof of Thm.~\ref{thm:main}]
    Take $f : X' \to X$ as in Prop.~\ref{prop:alterations}. By 
    Thm.~\ref{thm:ESS-main} and Prop.~\ref{prop:proper-effective-descent} the map $f$ satisfies the assumptions for Prop.~\ref{prop:fin-gen-pres-descent}\ref{prop:pt:fin-gen-descent}. This shows that $\piet(X)$ is finitely generated, for any scheme $X$ that is proper over $\Spec(k)$. In fact, finite generation was already proven in \cite[Exp.~X, Thm.~2.9]{SGA1}, and we included the argument here for the convenience of the reader. 
    
    We are going to apply finite generation to the connected components of $X'' = X' \times_X X'$, which are connected proper schemes over $\Spec(k)$. Indeed, using Thm.~\ref{thm:ESS-main} (this time crucially!) and Prop.~\ref{prop:proper-effective-descent} again,     the map $f$ now satisfies the assumptions of Prop.~\ref{prop:fin-gen-pres-descent}\ref{prop:pt:fin-pres-descent}. This shows that $\piet(X)$ is finitely presented. 
\end{proof}

\section{More general base fields}

  Similarly to \cite[\S5]{EsnShuSri}, our main result extends to more arithmetic settings. We thank Peter Haine for  essentially suggesting the following corollary.

\begin{corollary}
Let $X$ be a connected scheme that is proper over $\Spec(k)$ for a field $k$. Then  $\piet(X)$ is finitely presented if and only if the absolute Galois group $\Gal_k$ is finitely presented.
\end{corollary}
\begin{proof}
We may assume $X$ is reduced and thus $k' = \mathrm{H}^0(X,\calO_X)$ is a finite field extension of $k$. Let $\bar k$ be an algebraic closure of $k$ containing $k'$. Then $X \to \Spec(k')$ being the Stein factorization of $X \to \Spec(k)$ implies that $\bar X = X \times_{k'} \bar k$ is connected. By Thm.~\ref{thm:main}, the group $\pi_1(\bar X)$ is finitely presented. The fundamental exact sequence \cite[Exp.~IX, Thm.~6.1]{SGA1} 
\[
1 \to \pi_1(\bar X) \to \pi_1(X) \to \Gal_{k'} \to 1
\]
shows that $\pi_1(X)$ is finitely presented if and only if $\Gal_{k'}$ is finitely presented. The latter is equivalent to $\Gal_k$ being finitely presented, see for example \cite[Prop.~2.3]{EsnSriSti}.
\end{proof}

\begin{rmk}
  Examples of fields $k$ with finitely presented $\Gal_k$ include: fields algebraic over a finite field, local $p$-adic fields, $\mathbb{R}$ and more generally real closed fields, $K((t))$ for a field of characteristic $0$ with $\Gal_K$ of finite presentation, and by \cite[Thm.~5.1]{Jarden} for a hilbertian field $k$, probabilistically almost always (for the Haar measure on $\Gal_k$) the fixed field $k^\Sigma$ in the separable closure $\bar k$ of a finite subset $\Sigma \subset \Gal_k$.
\end{rmk}
\bibliographystyle{alphaSGA}
\bibliography{fp-pi1-biblio.bib}

\begin{thebibliography}{SGA71}

\bibitem[dJ96]{deJong-alterations}
Aise~Johan de~Jong.
\newblock Smoothness, semi-stability and alterations.
\newblock {\em IHES Publ. Math.}, 83:51--93, 1996.

\bibitem[ESS22]{EsnShuSri}
H\'{e}l\`ene Esnault, Mark Shusterman, and Vasudevan Srinivas.
\newblock Finite presentation of the tame fundamental group.
\newblock {\em Selecta Math. (N.S.)}, 28(2):Paper No. 37, 19, 2022.

\bibitem[ESS23]{EsnSriSti}
H{\'e}lene Esnault, Vasudevan Srinivas, and Jakob Stix.
\newblock An obstruction to lifting to characteristic $0$.
\newblock {\em Algebraic Geometry}, 10(3):327--347, 2023.

\bibitem[Hir75]{HironakaTriangulations}
Heisuke Hironaka.
\newblock Triangulations of algebraic sets.
\newblock In {\em Algebraic geometry ({P}roc. {S}ympos. {P}ure {M}ath., {V}ol.
  29, {A}rcata, {C}alif., 1974)}, pages 165--185. Amer. Math. Soc., Providence,
  R.I., 1975.

\bibitem[Jar74]{Jarden}
Moshe Jarden.
\newblock Algebraic extensions of finite corank of {H}ilbertian fields.
\newblock {\em Israel J. Math.}, 18:279--307, 1974.

\bibitem[{\L}oj64]{LojasiewiczTriangulations}
Stanis{\l}aw {\L}ojasiewicz.
\newblock Triangulation of semi-analytic sets.
\newblock {\em Ann. Scuola Norm. Sup. Pisa Cl. Sci. (3)}, 18:449--474, 1964.

\bibitem[Lub01]{Lubotzky}
Alexander Lubotzky.
\newblock Pro-finite presentations.
\newblock {\em J. Algebra}, 242(2):672--690, 2001.

\bibitem[SGA1]{SGA1}
{\em Rev\^{e}tements \'{e}tales et groupe fondamental}.
\newblock Lecture Notes in Mathematics, Vol. 224. Springer-Verlag, Berlin-New
  York, 1971.
\newblock S\'{e}minaire de G\'{e}om\'{e}trie Alg\'{e}brique du Bois Marie
  1960--1961 (SGA 1), Dirig\'{e} par Alexandre Grothendieck. Augment\'{e} de
  deux expos\'{e}s de Mich{\`{e}}le Raynaud.

\bibitem[Shu22]{ShustermanBalanced}
Mark Shusterman.
\newblock Balanced presentations for fundamental groups of curves over finite
  fields.
\newblock {\em Mathematical Research Letters}, 29(4):1251--1259, 2022.

\bibitem[Sti06]{Stix-vK}
Jakob Stix.
\newblock A general {S}eifert-{V}an {K}ampen theorem for algebraic fundamental
  groups.
\newblock {\em Publ. Res. Inst. Math. Sci.}, 42(3):763--786, 2006.

\end{thebibliography}

\end{document}